\documentclass[11pt]{amsart}
\usepackage{amsmath}
\usepackage{amsfonts}
\usepackage{amssymb}
\newtheorem{theorem}{Theorem}
\theoremstyle{plain}

\newtheorem{lemma}{Lemma}

\numberwithin{equation}{section}
\numberwithin{lemma}{section}
\numberwithin{theorem}{section}
\numberwithin{corollary}{section}
\numberwithin{proposition}{section}
\numberwithin{remark}{section}

\begin{document}
	\title{Units in $F(C_n \times Q_{12})$ and $F(C_n \times D_{12})$}
	\author{Meena Sahai* and Sheere Farhat Ansari** \\ Department of Mathematics and Astronomy\\ Lucknow University \\Lucknow 226007, India.\\
E-mail: *meena\_sahai@hotmail.com \ 
**sheere\_farhat@rediffmail.com}
\date{}
	\maketitle

\begin{abstract}
Let $C_n$, $Q_n$ and $D_n$ be the cyclic group, the quaternion group  and the dihedral group of order $n$, respectively. The structures of the unit groups of the finite group algebras  $FQ_{12}$  and  $F(C_2 \times Q_{12})$ over a finite field $F$ have been studied in J. Gildea, F.  Monaghan (2011),  F. Monaghan (2012), G. Tang, Y. Gao (2011) and  G. Tang, Y. Wei, Y. Li (2014) whereas the structures of the unit groups of the finite group algebras  $FD_{12}$ and $F(C_2 \times D_{12})$ have been studied in J.  Gildea,  F.  Monaghan (2011),  N. Makhijani, R. K. Sharma, J. B. Srivastava (2016), F. Monaghan (2012), M.  Sahai, S. F. Ansari and G. Tang, Y. Gao (2011).  In this paper, we continue this study and establish the structures of the unit groups of the group algebras $F(C_n \times Q_{12})$ and  $F(C_n \times D_{12})$, over a finite field $F$ of characteristic $p$ containing $p^k$ elements. 


\end{abstract}

\section{\textbf{Introduction}}
	
Let $FG$ be the group algebra of a finite group $G$ over a finite field $F$ of characteristic $p$ having $q=p^k$ elements. Let $U(FG)$ be the  unit group of $FG$ and let $J(FG)$ be the Jacobson radical of $FG$. If $V=1+J(FG)$, then $U(FG) \cong V \rtimes U(FG/J(FG))$.
	
 If $K$ is a normal subgroup of $G$  then the natural group epimorphism $G \to G/K$ can be extended to an $F$-algebra epimorphism $FG \to F(G/K)$. The kernel of this epimorphism $\omega(K)$,  is the ideal of $FG$ generated by $\{k-1 \mid k \in K\}$. In particular, if $K=G$, then the epimorphism $\epsilon : FG \to F$ given by $\epsilon(\sum_{g \in G} a_gg)=\sum_{g\in G}a_g$ is called the augmentation mapping of $FG$ and the ideal $\omega(G)$ is called the augmentation ideal of $FG$. Clearly, $FG/\omega(G) \cong F^{*}$. 
	
Let $C_n$, $Q_{n}$ 
 and $D_{n}$ 
be the  cyclic group, the quaternion group and the dihedral group of order $n$, respectively.  Also let  $C_n^k$ be the direct product of $k$ copies of $C_n$. The structure of $U(\mathbb{Z}_2D_{2p})$  for an odd prime $p$ is described in \cite{K2}. This was  extended to a field containing $2^k$ elements in \cite{N1}. In \cite{N2},  the structure of the centre of the maximal $p$-subgroup of $U(FD_{2p^n})$  for $n\geq 2$ is discussed.  Further,  by using an established isomorphism between $FG$ and a certain ring of $n \times n$ matrices in conjunction with other techniques, Gildea~\cite{J3} has obtained  the order of $U(FD_{2p^n})$ for an odd prime $p$ as $p^{{2k}(p^n-1)}(p^k-1)^2$  whereas in~\cite{J14}, he has proved  that  the centre of the maximal $p$-subgroup of $U(FD_{2p})$ is  $C_p^{k(p+1)/2}$. The structures of the unit groups of  $FQ_{12}$  and  $F(C_2 \times Q_{12})$
  have been studied in \cite{J2,  FM, T1, T2}. Also, the  unit groups of  $FD_{12}$  and $F(C_2 \times D_{12})$ have been studied in \cite{J2, N5, FM, sh12, T1}. In this paper, we continue this investigation  and establish the structures of  $U(F(C_n \times G))$ for $G=Q_{12}$ and $D_{12}$. 
	
	Throughout the paper,  
	 $F_n$ is the extension field of $F$ of degree $n$ and  $GL(n, F)$ is the general linear group of degree $n$ over $F$. For coprime integers $l$ and $m$, $ord_m(l)$ denotes the  multiplicative order of $l$ modulo $m$.

It is well known that, if $G$ and $H$ are groups, then $F(G \times H) \cong (FG)H$, the group ring of $H$ over the ring $FG$, see \cite[Chap 3, Page 134]{Miles}. This result will be used frequently. Now we state here some of the Lemmas  needed for our work.
	\begin{lemma}{\upshape\cite[Theorem 2.1]{N4}}\label{40}
		Let $F$ be a field of characteristic $p$ having $q=p^k$ elements. If $(n, p)=1$, where $n \in \mathbb{N}$, then
		$$FC_n \cong F \oplus \Big( \oplus_{{l>1}, ~{l|n}}F_{d_l}^{e_l}\Big),$$
		where $d_l=ord_l(q)$ and $e_l=\frac{\phi(l)}{d_l}$.
	\end{lemma}


\begin{lemma}{\upshape\cite[Lemma 3.3]{T1}}\label{9*}
	Let $F$ be a finite field of characteristic $p$ with $|F|=q=p^k$. If $p \neq 2$, then
	\begin{center}
		$FC_4 \cong
		\begin{cases}
		F^4, & \text{if $p \equiv 1$ mod $4$ or $n$ is even;}\\
		 F^2 \oplus F_2, & \text{if $p \equiv -1$ mod $4$ and $n$ is odd.}
		\end{cases}$
	\end{center}
\end{lemma}

\begin{lemma}\cite[Lemma 2.3]{N4}\label{11} 
	Let $F$ be a finite field of characteristic $p$ with $|F|=q=p^k$. Then
	\begin{center}
		$U(FC_{p^n}) \cong 
		\begin{cases}
		C_p^{(p-1)k} \times C_{p^k-1}, & \text{if $n=1$;}\\
	\prod_{s=1}^{n}C_{p^s}^{h_s} \times C_{p^k-1}, & \text{otherwise,}
		\end{cases}$
	\end{center}
	where $h_n=k(p-1)$ and $h_s=kp^{n-s-1}(p-1)^2$, for all s, $1 \leq  s < n$.
\end{lemma}

\begin{lemma}\cite[Lemma 3.2]{sh4}\label{5} Let $F$ be a finite field of characteristic $p$ with $|F|=q=p^k$. If $p \neq 2$, then
	$$U(FC_2^n) \cong C_{q-1}^{{2^n}}.$$
\end{lemma}



\section{Units in $F(C_n \times Q_{12})$}
 We shall use the following presentation for $C_n \times Q_{12}$:
	$$C_n \times Q_{12} =\langle x, y, z \mid x^{3}=y^4=z^n=1, xy=yx^2, xz=zx, yz=zy\rangle.$$
	\begin{theorem} Let $F$ be a finite field of characteristic $2$ containing $q=2^k$ elements and let $G=C_n \times Q_{12}$. If  $n$ is odd, then
		$$U(FG) \cong (C_2^{5nk} \times C_4^{nk})
		\rtimes \Bigg(\bigg(C_{q-1} \times GL(2, F)\bigg) \times \bigg(\prod_{l>1, ~l|n}\big(C_{q^{d_l}-1} \times GL(2, F_{d_l})\big)^{e_l}\bigg)\Bigg).$$
	\end{theorem}
	\begin{proof}
		Since $n$ is odd, $FC_n$ is semisimple. Thus by Lemma~\ref{40},
		\begin{align*}
		FG &\cong (FC_n)Q_{12},\\
		&\cong \Big(F \oplus  \big(\oplus_{l>1,~l|n}F_{d_l}^{e_l} \big)\Big)Q_{12},\\
		&\cong FQ_{12} \oplus\Big( \oplus_{l>1,~l|n}(F_{d_l}Q_{12})^{e_l}\Big).
		\end{align*}
		Now by \cite[Theorem 3.2]{T2},
		$U(FQ_{12}) \cong (C_2^{5k} \times C_4^{k}) \rtimes \big(C_{q-1} \times GL(2, F)\big)$
		and  so
		$$U(F_{d_l}Q_{12})^{e_l}\cong (C_2^{5\phi(l)k} \times C_4^{\phi(l)k}) \rtimes \big(C_{q^{d_l}-1} \times GL(2, F_{d_l})\big)^{e_l}.$$
		As $\sum_{{l|n}}\phi(l)=n$, so
			$$U(FG) \cong (C_2^{5nk} \times C_4^{nk})
		\rtimes \Bigg(\bigg(C_{q-1} \times GL(2, F)\bigg) \times \bigg(\prod_{l>1, ~l|n}\big(C_{q^{d_l}-1} \times GL(2, F_{d_l})\big)^{e_l}\bigg)\Bigg).$$
	\end{proof}	
	
	\begin{theorem} Let $F$ be a finite field of characteristic $3$ containing $q=3^k$ elements and let $G=C_n \times Q_{12}$. Then
		$$U(FG) \cong (C_3^{6nk} \rtimes C_3^{2nk}) \rtimes U\big(F(C_n \times C_4)\big).$$
		
	\end{theorem}
	\begin{proof}	Let $K = \langle x \rangle$. Then $G/K \cong H =
		\langle  y, z \rangle=C_n \times C_4$. Thus from the ring epimorphism $ FG \rightarrow FH$ given by  
		\begin{align*}
		\sum_{l=0}^{n-1}\sum_{j=0}^{3}\sum_{i=0}^2a_{i+3j+12l}x^iy^jz^l \mapsto  \sum_{l=0}^{n-1}\sum_{j=0}^{3}\sum_{i=0}^2a_{i+3j+12l}y^jz^l,	
		\end{align*}
		we get a	group epimorphism $\theta : U(FG) \to U(FH)$. 
		
		Further, from the inclusion map $ FH \rightarrow FG$, we have  $i : U(FH) \rightarrow U(FG)$ 
		such that $\theta i =1_{U(FH)}$. Therefore $U(FG)$ is a split extension of $U(FH)$ by  $V=ker(\theta) = 1+\omega(K)$. Hence 
		$$U(FG) \cong V \rtimes U(FH).$$
		
		Now, let $u=\sum_{l=0}^{n-1}\sum_{j=0}^{3}\sum_{i=0}^2a_{i+3j+12l}x^iy^jz^l\in U(FG)$, then $u \in V$ if and only if $\sum_{i=0}^2a_i=1$ and $\sum_{i=0}^2a_{i+3j}=0$ for $j=1, 2, \dots, (4n-1)$. Therefore  $$V=\Big\{1+\sum_{l=0}^{n-1}\sum_{j=0}^{3}\sum_{i=1}^2(x^i-1)b_{i+2j+8l}y^jz^l \mid b_i \in F\Big\}$$
		and $|V|=3^{8nk}$. Since $\omega(K)^3 = 0$, $V^3 = 1$. We now study the structure of $V$ in the following steps:\\
		
		{\bf Step 1:} $C_V(x)=\{v\in V \mid vx=xv\}\cong C_3^{6nk}$.\\
		
		If $v=1+ \sum_{l=0}^{n-1}\sum_{j=0}^{3}\sum_{i=1}^2(x^i-1)b_{i+2j+8l}y^jz^l\in C_V(x)=\{v\in V \mid vx=xv\}$, then
		$vx-xv=\widehat{x}\sum_{l=0}^{n-1}\big((b_{3+8l}-b_{4+8l})y+(b_{7+8l}-b_{8+8l})y^3\big)z^i$.
		Thus $v\in C_V(x)$ if and only if $b_{j+8l}=b_{1+j+8l}$  for $j=3, 7$ and $l=0, 1, \dots, n-1$.
		Hence
		\begin{align*}
		C_V(x)=&\Big\{1+\sum_{l=0}^{n-1}\sum_{j=0}^1\sum_{i=1}^2(x^{i}-1)c_{i+2j+4l}y^{2j}z^{l}\\&+\widehat{x}\sum_{l=0}^{n-1}\sum_{j=0}^1c_{4n+nj+l+1}y^{2j+1}z^l \mid c_i \in F\Big\}.
		\end{align*}
		So $C_V(x)$ is an abelian subgroup of $V$ and $|C_V(x)| = 3^{6nk}$. Therefore $C_V(x) \cong C_3^{6nk}$.\\
		
		{\bf Step 2:}
		Let $T$ be the subset of $V$ consisting of elements of the form
		$$1+\sum_{j=0}^{n-1}\Big(\widehat{x}(t_{j1}+t_{j2}y^2)+(x+2x^2)(t_{j3}y+t_{j4}y^3)\Big)z^{j},$$
		where $t_{j_{i}} \in F$. Then $T$ is an abelian subgroup of $V$ and 
		$ T\cong C_3^{4nk}.$
		
		Let 	$$t_1=1+\sum_{j=0}^{n-1}\Big(\widehat{x}(r_{j1}+r_{j2}y^2)+(x+2x^2)(r_{j3}y+r_{j4}y^3)\Big)z^{j} \in T$$	
		and 
		$$t_2=1+\sum_{j=0}^{n-1}\Big(\widehat{x}(s_{j1}+s_{j2}y^2)+(x+2x^2)(s_{j3}y+s_{j4}y^3)\Big) z^{j}\in T.$$
		Then
		\begin{align*}
		t_1t_2=&1+\sum_{j=0}^{n-1}\Big(\widehat{x}\big((r_{j1}+s_{j1}+\gamma_{1})+(r_{j2}+s_{j2}+\gamma_{2})y^2\big)\\&+(x+2x^2)\big((r_{j3}+s_{j3})y+(r_{j4}+s_{j4})y^3\big)\Big)z^{j}\in T,
		\end{align*}
		where 
		\begin{align*}
		\gamma_{1}&=2\sum_{i=0}^{n-1}(r_{j3}s_{i4}+r_{j4}s_{i3})z^i,\\
		\gamma_{2}&=2\sum_{i=0}^{n-1}(r_{j3}s_{i3}+r_{j4}s_{i4})z^i.
		\end{align*}
		So $T$ is an abelian subgroup of $V$ and $|T| = 3^{4nk}$. Therefore $T\cong C_3^{4nk}$.
		
		Now, let	
		\begin{align*}
		c&=1+\sum_{l=0}^{n-1}\sum_{j=0}^1\sum_{i=1}^2(x^{i}-1)c_{i+2j+4l}y^{2j}z^{l}\\&+\widehat{x}\sum_{l=0}^{n-1}\sum_{j=0}^1c_{4n+nj+l+1}y^{2j+1}z^l
		\in C_V(x)
		\end{align*}	
		and 
		$$t= 1+\sum_{j=0}^{n-1}\Big(\widehat{x}(t_{j1}+t_{j2}y^2)+(x+2x^2)(t_{j3}y+t_{j4}y^3)\Big)z^{j}\in T.$$ Then
		\begin{align*}
		t^{-1}&= 1+2\sum_{j=0}^{n-1}\Big(\widehat{x}(t_{j1}+t_{j2}y^2)+(x+2x^2)(t_{j3}y+t_{j4}y^3)\Big)z^{j}\\&~~~~+2\sum_{j=0}^{n-1}\widehat{x}\Big((t_{j3}^2+t_{j4}^2)y^2+2
		t_{j3}t_{j4}\Big)z^{2j}
		\end{align*}
		and
		\begin{align*}
		c^t&= c+\widehat{x}\sum_{i=0}^{n-1}\sum_{j=0}^{n-1}\Big((c_{1+4i}-c_{2+4i})t_{j3}+(c_{3+4i}-c_{4+4i})t_{j4}\Big)yz^{i+j}\\&~~~~+\widehat{x}\sum_{i=0}^{n-1}\sum_{j=0}^{n-1}\Big((c_{1+4i}-c_{2+4i})t_{j4}+(c_{3+4i}-c_{4+4i})t_{j3}\Big)yz^{i+j}.
		\end{align*}
		Clearly, $c^t \in C_V(x)$. Thus $T$ normalizes $C_V(x)$. Now, if $U=C_V(x) \cap T$, then
		$$U=\Big\{1+\widehat{x}\sum_{j=0}^{n-1}(t_{j1}+t_{j2}y^2)z^{j} \mid t_{ji} \in F\Big\}\cong C_3^{2nk}.$$
		So for some subgroup $W \cong C_3^{2nk}$ of $T$, we have $T=U \times W$, $C_V(x) \cap W=1$ and $|C_V(x)W|=|V|=3^{8nk}$. Hence $V \cong C_V(x) \rtimes W \cong C_{3}^{6nk} \rtimes C_{3}^{2nk}$.
	\end{proof}
	For $U\big(F(C_n \times C_4)\big)$, we prove the following:
	\begin{theorem}
		Let $F$ be a finite field of characteristic $3$ containing $q=3^k$ elements and let $H=C_n \times C_4$, where $n=3^rs$ such that $r\geq 0$ and $(3, s)=1$. Then $U(FH)$ is isomorphic to 
		\begin{enumerate}
			\item If $3 \nmid n$, then 
			\begin{enumerate}
				\item $C_{q-1}^4 \times\big( \prod_{{l>1}, ~{l|n}}C_{q^{d_l}-1}^{4e_l}\big)$, if $q \equiv 1$ mod $4$;
				\item $C_{q-1}^2 \times C_{q^{2}-1} \times \big( \prod_{{l>1}, ~{l|n}}(C_{q^{d_l}-1}^{2e_l}\times C_{q^{d'_l}-1}^{e'_l})\big)$,  if   $q \equiv -1$ mod $4$.
			\end{enumerate}
			\item If $3|n$, then
			\begin{enumerate}
				\item $C_{q-1}^4  \times \big(\prod_{{l>1}, ~{l|s}}C_{q^{d_l}-1}^{4e_l}\big) \times \big(\prod_{t=1}^rC_{3^t}^{4sn_t}\big)$,  if   $q \equiv 1$ mod $4$;
				\item $C_{q-1}^2 \times C_{q^{2}-1} \times \big(\prod_{{l>1}, ~{l|s}}(C_{q^{d_l}-1}^{2e_l}\times C_{q^{d'_l}-1}^{e'_l}) \big)\times \big(\prod_{t=1}^r(C_{3^t}^{2sn_t}\times C_{3^t}^{sn'_t})\big)$,  if $q \equiv -1$ mod $4$;
				
				where $d'_l=ord_l(q^2)$, $e'_l=\frac{\phi(l)}{d'_l}$, $n_r=2k$, $n_t=4.3^{r-t-1}k$,  for all $1\leq t < r$ and $n'_t =2n_t$, for all $1 \leq t \leq r$.
			\end{enumerate}
		\end{enumerate}
	\end{theorem}
	\begin{proof} As $FH\cong (FC_4)C_n$, so using Lemma~\ref{9*}, we have
		$$FH\cong 
		\begin{cases}
		(FC_n)^4, & \text{if $q\equiv 1$ mod $4$;}\\
		(FC_n)^2 \oplus F_2C_n, &\text{if $q\equiv -1$ mod $4$.}
		\end{cases}$$
		\begin{enumerate}
			\item If $3\nmid n$, then  by Lemma~\ref{40},
			$$FC_n\cong F \oplus\Big( \oplus_{l>1,~l|n}F_{d_l}^{e_l}\Big)$$
			and so
			$$F_2C_n \cong F_2 \oplus \Big(\oplus_{l>1,~l|n}F_{d'_l}^{e'_l}\Big),$$
			where $d'_l=ord_l(q^2)$ and $e'_l=\frac{\phi(l)}{d'_l}$.
			Hence
			$$FH\cong 
			\begin{cases}
			F^4 \oplus (\oplus_{l>1,~l|n}F_{d_l}^{4e_l}), & \text{if $q\equiv 1$ mod $4$;}\\
			F^2 \oplus F_2 \oplus\big( \oplus_{l>1,~l|n}(F_{d_l}^{2e_l} \oplus F_{d'_l}^{e'_l})\big), &\text{if $q\equiv -1$ mod $4$.}
			\end{cases}$$
			It is obvious that
			$$d_l'=\begin{cases}
			d_l/2, & \text{if $d_l$ is even;}\\
			d_l, & \text{if $d_l$ is odd}.
			\end{cases}$$
			Also
			$$e_l'=\begin{cases}
			2e_l, & \text{if $d_l$ is even;}\\
			e_l, & \text{if $d_l$ is odd}.
			\end{cases}$$
			
			\item If $3|n$, then by Lemma~\ref{40},
			\begin{align*}
			FC_n &\cong (FC_s)C_{3^r},\\
			&\cong \big(F \oplus ( \oplus_{l>1,~l|s}F_{d_l}^{e_l} )\big)C_{3^r},\\
			&\cong FC_{3^r} \oplus \big(\oplus_{l>1,~l|s}(F_{d_l}C_{3^r})^{e_l}\big).
			\end{align*}
			By Lemma~\ref{11},
			$$U(FC_{3^r}) \cong C_{3^k-1} \times \Big(\prod_{t=1}^rC_{3^t}^{n_t}\Big)$$
			where $n_r=2k$, $n_t=4.3^{r-t-1}k$.	Thus $$U(F_{d_l}C_{3^r})^{e_l} \cong C_{3^{d_lk}-1}^{e_l} \times \Big( \prod_{t=1}^rC_{3^t}^{\phi(l)n_t}\Big).$$
			Since $\sum_{{l|s}}\phi(l)=s$,
			$$U(FC_n) \cong C_{3^k-1} \times \Big(\prod_{{l>1}, ~{l|s}} C_{3^{d_lk}-1}^{e_l}\Big) \times \Big(\prod_{t=1}^r C_{3^t}^{sn_t}\Big)$$
			and 
			$$U(F_2C_n) \cong C_{3^{2k}-1} \times \Big(\prod_{{l>1}, ~{l|s}} C_{3^{d'_lk}-1}^{e'_l}\Big) \times \Big(\prod_{t=1}^r C_{3^t}^{sn'_t}\Big),$$
			where $n'_t = 2n_t$, for all $1\leq t \leq r$.
			Hence the claim.
		\end{enumerate}	
	\end{proof}
	\begin{theorem}	Let $F$ be a finite field of characteristic $p>3$ containing $q=p^k$ elements and let $G=C_n \times Q_{12}$ where $n=p^rs$, $r\geq 0$ such that $(p, s)=1$. If $V=1+J(FG)$, then 
		$U(FG)/V$ is isomorphic to
		\begin{enumerate}
			\item $C_{q-1}^4 \times GL(2, F)^2 \times \Big(\prod_{l>1, ~l|s}\big(C_{q^{d_l}-1}^{4} \times GL(2, F_{d_l})^{2}\big)^{e_l}\Big)$, if $q\equiv 1, 5$ mod $12$;
			\item	$C_{q-1}^2 \times C_{q^2-1} \times GL(2, F)^2 \times \Big(\prod_{l>1, ~l|s}\big(C_{q^{d_l}-1}^{2} \times C_{q^{2d_l}-1}\times GL(2, F_{d_l})^{2}\big)^{e_l}\Big)$, if $q\equiv -1, -5$ mod $12$;
		\end{enumerate}
		where $V$ is a group of exponent $p^r$ and order $p^{12sk(p^r-1)}$. 
	\end{theorem}
	\begin{proof} Let $K=\langle z^s \rangle$. Then $G/K \cong H = C_s\times Q_{12}$. If $\theta : FG \rightarrow FH$ is the canonical ring epimorphism, then 
		$J(FG)=ker(\theta)$, $FG/J(FG) \cong FH$ and $dim_{F}(J(FG))=12s(p^r-1)$. Hence 
		$U(FG) \cong V \rtimes U(FH)$,
		where $V= 1+J(FG)$. Clearly, exponent of $V=p^r$ and $|V|=p^{12sk(p^r-1)}$.
		
		By Lemma~\ref{40}, 
		\begin{align*}
		FH &\cong (FC_{s})Q_{12},\\
		&\cong \big(F\oplus (\oplus_{{l>1, ~l|s}}F_{d_l}^{e_l})\big)Q_{12},\\
		&\cong FQ_{12} \oplus \big( \oplus_{l>1, ~l|s}(F_{d_l}Q_{12})^{e_l}\big).
		\end{align*}
		Now, by \cite[Theorem 4.2]{T1},
		$$FQ_{12} \cong 
		\begin{cases}
		F^4 \oplus M(2, F)^2, & \text{ if $q\equiv 1, 5$ mod $12$;}\\
		F^2 \oplus F_2 \oplus M(2, F)^2, &\text{if $q \equiv -1, -5$ mod $12$.}
		\end{cases}$$ 
		and so
		$$(F_{d_l}Q_{12})^{e_l} \cong 
		\begin{cases}
		F_{d_l}^{4e_l}\oplus M(2, F_{d_l})^{2e_l}, & \text{ if $q\equiv 1, 5$ mod $12$;}\\
		F_{d_l}^{2e_l} \oplus F_{2d_l}^{e_l} \oplus M(2, F_{d_l})^{2e_l}, &\text{if $q \equiv -1, -5$ mod $12$.}
		\end{cases}$$ 
	\end{proof}
	In the above theorem, if $r=0$, then we have the unit group of the semisimple group algebra $FG$ given by 
	\begin{enumerate}
		\item $U(FG) \cong C_{q-1}^4 \times GL(2, F)^2 \times\Big( \prod_{l>1, ~l|n}\big(C_{q^{d_l}-1}^{4} \times GL(2, F_{d_l})^{2}\big)^{e_l}\Big)$, if $q\equiv 1, 5$ mod $12$.
		\item $U(FG) \cong C_{q-1}^2 \times C_{q^2-1} \times GL(2, F)^2 \times \Big(\prod_{l>1, ~l|n}\big(C_{q^{d_l}-1}^{2} \times  C_{q^{2d_l}-1}\times GL(2, F_{d_l})^{2}\big)^{e_l}\Big)$, if $q\equiv -1, -5$ mod $12$.
	\end{enumerate}

\section{Units in  $F(C_n \times D_{12})$}
We shall use the following presentation for $C_n \times D_{12}$:
$$C_n \times D_{12} = \langle x, y, z \mid x^{6}=y^2=z^n=1, yx=x^{5}y, xz=zx, yz=zy\rangle.$$
\begin{theorem} Let $F$ be a finite field of characteristic $2$ containing $q=2^k$ elements and let $G=C_n \times D_{12}$. If  $n$ is odd, then
	$$U(FG) \cong C_2^{7nk} \rtimes \Bigg( C_{q-1} \times GL(2, F) \times \Big(\prod_{l>1, ~l|n}\big(C_{q^{d_l}-1} \times GL(2, F_{d_l})\big)^{e_l}\Big)\Bigg).$$
\end{theorem}
\begin{proof}
	Since $n$ is odd, $FC_n$ is semisimple. Thus by Lemma~\ref{40},
	\begin{align*}
	FG &\cong (FC_n)D_{12},\\
	&\cong \big(F \oplus  (\oplus_{l>1,~l|n}F_{d_l}^{e_l} )\big)D_{12},\\
	&\cong FD_{12} \oplus \big(\oplus_{l>1,~l|n}(F_{d_l}D_{12})^{e_l}\big).
	\end{align*}
		Now by \cite[Theorem 2.6]{N5},
	$$U(FD_{12}) \cong C_2^{7k} \rtimes \Big(C_{q-1} \times GL(2, F)\Big)$$
	and so
	$$U(F_{d_l}D_{12})^{e_l} \cong C_2^{7\phi(l)k} \rtimes \Big(C_{q^{d_l}-1}^{e_l} \times GL(2, F_{d_l})^{e_l}\Big).$$
	Since
	$\sum_{{l|n}}\phi(l)=n$, 
	$$U(FG) \cong C_2^{7nk} \rtimes \Bigg(C_{q-1} \times GL(2, F) \times \Big(\prod_{{l>1}, ~{l|n}} \big(C_{q^{d_l}-1} \times GL(2, F_{d_l})\big)^{e_l}\Big)\Bigg).$$
\end{proof}	
\begin{theorem} Let $F$ be a finite field of characteristic $3$ containing $q=3^k$ elements and let $G=C_n \times D_{12}$. Then
	$$U(FG) \cong (C_3^{6nk} \rtimes C_3^{2nk}) \rtimes U\big(F(C_n \times C_2^2)\big).$$
\end{theorem}
\begin{proof} Let $K = \langle x^{2} \rangle$. Then $G/K \cong H =
	\langle x^3, y, z \rangle=C_2 \times C_2 \times C_n$. Thus from the ring epimorphism $ FG \rightarrow FH$ given by  
	\begin{align*}
	&\sum_{l=0}^1\sum_{j=0}^{n-1}\sum_{i=0}^2x^{2i}(a_{i+6(j+nl)}+x^{3}a_{i+6(j+nl)+3})y^lz^j \mapsto \\&	\sum_{l=0}^1\sum_{j=0}^{n-1}\sum_{i=0}^2(a_{i+6(j+nl)}+x^{3}a_{i+6(j+nl)+3})y^lz^j
	\end{align*}
	we get a	group epimorphism $\theta : U(FG) \to U(FH)$. 
	
	Further, from the inclusion map $i : FH \rightarrow FG$, we have  $i : U(FH) \rightarrow U(FG)$ 
	such that $\theta i =1_{U(FH)}$. Therefore $U(FG)$ is a split extension of $U(FH)$ by $V = ker(\theta)=1+\omega(K)$. Hence $$U(FG) \cong V \rtimes U(FH).$$
	
	Let $u=\sum_{l=0}^1\sum_{j=0}^{n-1}\sum_{i=0}^2x^{2i}(a_{i+6(j+nl)}+x^{3}a_{i+6(j+nl)+3})y^lz^j \in U(FG)$, then $u \in V$ if and only if $\sum_{i=0}^2a_i=1$ and $\sum_{i=0}^2a_{i+3j}=0$ for $j=1, 2, \dots, (4n-1)$. Therefore  $$V=\Big\{1+\sum_{l=0}^{1}\sum_{j=0}^{n-1}\sum_{i=1}^2(x^{2i}-1)(b_{i+4(j+nl)}+x^{3}b_{i+4(j+nl)+2})y^lz^j \mid b_i \in F\Big\}$$
	and $|V|=3^{8nk}$. Since $\omega(K)^3 = 0$, $V^3 = 1$. We now study the structure of $V$ in the following steps:\\
	
	{\bf Step 1:} $C_V(x^{2})=\{v\in V \mid vx^{2}=x^{2}v\}\cong C_3^{6nk}$.\\
	
	If $v=1+ \sum_{l=0}^{1}\sum_{j=0}^{n-1}\sum_{i=1}^2(x^{2i}-1)(b_{i+4(j+nl)}+x^{3}b_{i+4(j+nl)+2})y^lz^j\in C_V(x^{2})=\{v\in V \mid vx^{2}=x^{2}v\}$, then
	$vx^{2}-x^{2}v=\widehat{x^{2}}\sum_{j=0}^{n-1}\big((b_{1+4(j+n)}-b_{2+4(j+n)})+x^{3}(b_{3+4(j+n)}-b_{4+4(j+n)})\big)yz^j$.
	Thus $v\in C_V(x^{2})$ if and only if $b_{i+4(j+n)}=b_{1+i+4(j+n)}$  for $j=0, 1, \dots, n-1$ and $i
	=1, 3$.
	Hence
	\begin{align*}
	C_V(x^{2})=&\Big\{1+\sum_{j=0}^{n-1}\sum_{i=1}^2(x^{2i}-1)(c_{i+4j}+x^3c_{i+4j+2})z^{j}\\&+\widehat{x^{2}}\sum_{j=0}^{n-1}\sum_{i=0}^1c_{n(i+4)+j+1}x^{3i}yz^j \mid c_i \in F\Big\}.
	\end{align*}
	So $C_V(x^{2})$ is an abelian subgroup of $V$ and $|C_V(x^{2})| = 3^{6nk}$. Therefore $C_V(x^{2}) \cong C_3^{6nk}$.
	
	{\bf Step 2:}
	Let $S$ be the subset of $V$ consisting of elements of the form
	$$1+\sum_{j=0}^{n-1}x^{2}(1-x^2)(s_{j_{1}}+s_{j_{2}}x^3)(1+y)z^j,$$
	where $s_{j_{1}}, s_{j_{2}} \in F$. Then $S$ is an abelian subgroup of $V$ and 
	$S\cong C_3^{2nk}.$
	
	Let 	$$s_1 =1+\sum_{j=0}^{n-1}x^{2}(1-x^2)(r_{j_{1}}+r_{j_{2}}x^3)(1+y)z^j \in S$$	
	and 
	$$s_2=1+\sum_{j=0}^{n-1}x^{2}(1-x^2)(t_{j_{1}}+t_{j_{2}}x^3)(1+y)z^j\in S.$$
	Then
	\begin{align*}
	s_1s_2=1+\sum_{j=0}^{n-1}x^{2}(1-x^2)\Big((r_{j_1}+t_{j_{1}})+(r_{j_{2}}+t_{J_2})x^3\Big)(1+y)z^j\in S.
	\end{align*}
	So $S$ is an abelian subgroup of $V$ and $|S| = 3^{2nk}$. Therefore $S\cong C_3^{2nk}$.
	
	Now, let	
	\begin{align*}
	c=&1+\sum_{j=0}^{n-1}\sum_{i=1}^2(x^{2i}-1)(c_{i+4j}+x^3c_{i+4j+2})z^{j}\\&+\widehat{x^{2}}\sum_{j=0}^{n-1}\sum_{i=0}^1c_{n(i+4)+j+1}x^{3i}yz^j\in C_V(x^2)
	\end{align*}	
	and 
	$$s= 1+\sum_{j=0}^{n-1}x^{2}(1-x^2)(s_{j_{1}}+s_{j_{2}}x^3)(1+y)z^j\in S.$$  Then
	\begin{align*}
	c^s= c+\widehat{x^2}(\gamma_{1}+\gamma_{2}x^3)y,
	\end{align*}
	where 
	\begin{align*}
	\gamma_{1}=&\sum_{j=0}^{n-1}\sum_{i=0}^{n-1}\Big(s_{j_{1}}(c_{1+4i}-c_{2+4i})+s_{j_{2}}(c_{3+4i}-c_{4+4i})\Big)z^{i+j},\\
	\gamma_{2}=&\sum_{j=0}^{n-1}\sum_{i=0}^{n-1}\Big(s_{j_{1}}(c_{3+4i}-c_{4+4i})+s_{j_{2}}(c_{1+4i}-c_{2+4i})\Big)z^{i+j}.
	\end{align*}
	Clearly, $c^s \in C_V(x^2)$. Thus $S$ normalizes $C_V(x^{2})$. Since $C_V(x^{2}) \cap S=1$,  $|C_V(x^{2})S|=3^{8nk} = |V|$. Therefore $$V=C_V(x^{2})S\cong C_V(x^{2}) \rtimes S \cong C_3^{6nk} \rtimes C_3^{2nk}.$$
	Hence the claim.
\end{proof}
For $U\big(F(C_n \times C_2^2)\big)$, we prove the following:
\begin{theorem}
	Let $F$ be a finite field of characteristic $3$ containing $q=3^k$ elements and let $H=C_n \times C_2^2$, where $n=3^rs$ such that $r\geq 0$ and $(3, s)=1$. Then $U(FH)$ is isomorphic to 
	\begin{enumerate}
		\item $C_{q-1}^4 \times \big(\prod_{{l>1}, ~{l|n}}C_{q^{d_l}-1}^{4e_l}\big), \text { if } 3\nmid n;$
		\item $C_{q-1}^4 \times \big(\prod_{{l>1}, ~{l|s}}C_{q^{d_l}-1}^{4e_l}\big) \times \big(\prod_{t=1}^rC_{3^t}^{4sn_t}\big), \text { if }  3|n;$
		
		where $n_r=2k$ and $n_t=4k3^{r-t-1}$, for all t, $1 \leq t < r$.
	\end{enumerate}
\end{theorem}
\begin{proof}By Lemma~ \ref{5}, we have
	\begin{align*}
	FH&\cong (FC_2^2)C_n\cong (FC_n)^4.
	\end{align*}
	\begin{enumerate}
		\item If $3\nmid n$, i.e., if $r=0$, then  by Lemma~\ref{40},
		\begin{equation*}\label{E100}
		FC_n \cong F \oplus \Big(\oplus_{l>1,~l|n}F_{d_l}^{e_l}\Big).
		\end{equation*}
		Hence
		$$U(FH) \cong C_{3^k-1}^4 \times \Big(\prod_{{l>1}, ~{l|n}}C_{3^{d_lk}-1}^{4e_l}\Big).$$
		\item If $3|n$, i.e., if $r>0$, then by Lemma~\ref{40},
		\begin{align*}
		FC_n &\cong (FC_s)C_{3^r},\\
		&\cong\big(F \oplus ( \oplus_{l>1,~l|s}F_{d_l}^{e_l} )\big)C_{3^r},\\
		&\cong FC_{3^r} \oplus \big(\oplus_{l>1,~l|s}(F_{d_l}C_{3^r})^{e_l}\big).
		\end{align*}
			By Lemma~\ref{11},
		$$U(FC_{3^r}) \cong C_{3^k-1} \times\Big( \prod_{t=1}^rC_{3^t}^{n_t}\Big),$$
		where $n_r=2k$ and $n_t=4k3^{r-t-1}$, for all t, $1 \leq t < r$
		and $$U(F_{d_l}C_{3^r})^{e_l}\cong C_{3^{d_lk}-1}^{e_l} \times \Big(\prod_{t=1}^rC_{3^t}^{\phi(l)n_t}\Big).$$
		Since $\sum_{{l|s}}\phi(l)=s$, 
		$$U(FC_n) \cong C_{3^k-1} \times \Big(\prod_{{l>1}, ~{l|s}} C_{3^{d_lk}-1}^{e_l}\Big) \times \Big(\prod_{t=1}^r C_{3^t}^{sn_t}\Big)$$
		and hence
		$$U(FH) \cong C_{3^k-1}^4 \times \Big( \prod_{{l>1}, ~{l|s}} C_{3^{d_lk}-1}^{4e_l}\Big) \times \Big(\prod_{t=1}^r C_{3^t}^{4sn_t}\Big).$$
	\end{enumerate}	
\end{proof}
\begin{theorem}	Let $F$ be a finite field of characteristic $p>3$ containing $q=p^k$ elements and let $G=C_n \times D_{12}$, where $n=p^rs$, $r\geq 0$ such that $(p, s)=1$. If $V=1+J(FG)$, then 
	$$U(FG)/V\cong  C_{q-1}^4 \times GL(2, F)^2 \times \Big(\prod_{l>1, ~l|s}\big(C_{q^{d_l}-1}^{4} \times GL(2, F_{d_l})^{2}\big)^{e_l}\Big),$$
	where $V$ is a group of exponent $p^r$ and order $p^{12sk(p^r-1)}$.
\end{theorem}
\begin{proof} Let $K=\langle z^s \rangle$. Then $G/K \cong H = C_s\times D_{12}$. If $\theta : FG \rightarrow FH$ is the canonical ring epimorphism, then by \cite[Theorem 7.2.7 and Lemma 8.1.17]{Pass},
	$J(FG)=ker(\theta)$, $FG/J(FG) \cong FH$ and $dim_{F}(J(FG))=12s(p^r-1)$. Hence 
	$U(FG) \cong V \rtimes U(FH)$.
	Clearly, exponent of $V=p^r$ and $|V|=p^{12sk(p^r-1)}$.
	Now by Lemma~\ref{40}, 
	\begin{align*}
	FH &\cong (FC_{s})D_{12},\\
	&\cong \big(F\oplus (\oplus_{{l>1, ~l|s}}F_{d_l}^{e_l})\big)D_{12},\\
	&\cong FD_{12} \oplus \big(\oplus_{l>1, ~l|s}(F_{d_l}D_{12})^{e_l}\big).
	\end{align*}
	Now, by \cite[Theorem 4.3]{T1}, $FD_{12} \cong F^4 \oplus M(2, F)^2$. 
	Hence
	$$U(FH) \cong  C_{q-1}^4 \times GL(2, F)^2 \times \Big(\prod_{l>1, ~l|s}\big(C_{q^{d_l}-1}^{4} \times GL(2, F_{d_l})^{2}\big)^{e_l}\Big).$$
\end{proof}

In the above theorem, if $r=0$, then we have the unit group of the semisimple group algebra $FG$ given by 
$$U(FG)\cong C_{q-1}^4 \times GL(2, F)^2 \times \Big(\prod_{{l>1, ~l|n}} \big(C_{q^{d_l}-1}^{4}  \times GL(2, F_{d_l})^{2}\big)^{e_l}\Big).$$

\medskip

\noindent\textbf{Acknowledgements:} The financial assistance provided  to the first author (SFA) in the form of a Senior Research Fellowship from the University Grants Commission, India is gratefully acknowledged.

\end{document}